\documentclass{siamltex}
\usepackage{amsmath,amsfonts,amssymb}
\usepackage{latexsym}
\usepackage{graphicx,overpic}
\usepackage{subfigure,caption}
\usepackage{pgfplots}
\pgfplotsset{compat=newest}
\usepackage{color}
\usepackage{booktabs}
\usepackage{tikz}

\newcommand{\R}{\mathbb{R}}

\newcommand{\be}{\mathbf e}

\newcommand{\bu}{\mathbf u}
\newcommand{\bv}{\mathbf v}

\newcommand{\bx}{\mathbf x}

\newcommand{\cE}{\mathcal E}

\DeclareGraphicsExtensions{.pdf,.eps,.ps,.eps.gz,.ps.gz,.eps.Y}

\newcommand{\bfone}{\mathbf{1}}

\newtheorem{assumption}{Assumption}[section]

\newtheorem{remark}{Remark}[section]

\begin{document}

\title{Discretization error analysis for a radially symmetric harmonic map heat flow problem }
\author{Nam Anh Nguyen \thanks{Institut f\"ur Geometrie und Praktische  Mathematik, RWTH-Aachen
University, D-52056 Aachen, Germany (nguyen@igpm.rwth-aachen.de)} \and Arnold Reusken\thanks{Institut f\"ur Geometrie und Praktische  Mathematik, RWTH-Aachen
University, D-52056 Aachen, Germany (reusken@igpm.rwth-aachen.de)} 
}
\maketitle
%
\begin{abstract} 
In this paper we study the harmonic map heat flow problem for a radially symmetric case. The corresponding partial dfferential equation plays a key role in many analyses of  harmonic map heat flow problems. We consider  a basic discretization method for this problem, namely a second order finite difference discretization in space  combined with a semi-implicit Euler method in time. The semi-implicit Euler method results in a linear problem in each time step. We restrict to the regime of smooth solutions of the continuous problem and present an error analysis of this discretization method. This results in optimal order discretization error bounds (apart from a logarithmic term). We also present discrete energy estimates that mimic the decrease of the energy  of the continuous solution.  
\end{abstract}
\begin{keywords}  Harmonic map heat flow problem, discrete stability analysis, discretization error bounds, discrete energy estimates.
 \end{keywords}
\section{Introduction}
Let $\Omega \subset \Bbb{R}^N$, $N=2,3$,  be a Lipschitz domain and $S^2$ the unit sphere in $\Bbb{R}^3$. The harmonic  map heat flow (HMHF) problem is as follows. Given an initial condition $\bu_0: \Omega \to S^2$, determine $\bu(\cdot,t): \Omega \to S^2$ such that
\begin{equation} \label{HMHFeq}
 \partial_t\mathbf{u}= \Delta \mathbf{u} + \rvert \nabla \mathbf{u} \lvert^2 \mathbf{u}, \quad \bu(\cdot,0)= \bu_0, \quad \bu(\cdot,t)_{|\partial \Omega}=(\bu_0)_{|\partial \Omega}, \quad t \in (0,T]. 
\end{equation}
This problem is obtained as the $L^2$ gradient flow of the Dirichlet energy $\int_\Omega |\nabla \bv|^2 \, dx$ on vector fields $\bv: \Omega \to \Bbb{R}^3$ that satisfy a pointwise unit length constraint. Unit length minimizers of this Dirichlet energy are called harmonic maps.  The HMHF problem is  closely related to the Landau-Lifshitz-Gilbert (LLG) equation. The HMHF equation can be considered as the limit of the LLG equation where the precessional term vanishes and only damping is left \cite{Melcher2011}. 
Harmonic maps, HMHF and LLG equations have numerous  applications, for example, in the modeling of ferromagnetic materials or of liquid crystals, cf. e.g.~\cite{Lakshmanan2011,LiuLi2013Skyr,prohl2001,Heinze2011}.
We refer to \cite[Chapter III, Section 6]{Struwe08} and references therein for a further introduction to HMHF.

There is extensive mathematical literature in which topics related to well-posedness, weak formulations, regularity, blow-up phenomena and convergence of solutions of the HMHF problem to harmonic maps are studied, cf. e.g~\cite{Struwe85,Struwe08,Eells1964HarmonicMO,Hamilton1975HarmonicMO,KCChang89,Freire95,ChangDingYe1992,vandHout2001}.
 
 Early work on the development and analysis of numerical methods for HMHF or LLG problems is found in \cite{BarPro06,BaLuPr09,prohl2001,BartelsProhl2007,Alouges2008FEMLL, AlougesJaisson06}. In recent years, there has been a renewed interest in the numerical analysis of methods for this problem class \cite{KovacsLuebich,BaKoWa22,Gui22,BBPS24}.
 
 In this paper we study the HMHF problem for a specific \emph{radially symmetric} case. Assume $\Omega$ is the unit disk in $\R^2$ and assume the solution $\bu$  to be radially symmetric. Using polar coordinates on the disk a special type of solution  of \eqref{HMHFeq} is given by
 \begin{equation} \label{polcoord} \mathbf{u}(r,\psi,t) =
   \begin{pmatrix}  \cos{\psi} \sin{ u(r,t) } \\  \sin{\psi} \sin{ u(r,t) } \\ \cos{ u(r,t)}\end{pmatrix}, 
\end{equation}
and the partial differential equation in \eqref{HMHFeq} can be reduced to 
 \begin{equation} \label{PDE2}
  u_t  = u_{xx} + \frac{1}{x}u_x - \frac{\sin(2u)}{2 x^2} \quad \text{on}~ (0,1] \times [0,T],  \\
\end{equation}
  for the function $u(x,t)=u(r,t)$ on $[0,1] \times [0,T]$, as is shown in \cite[Section 2.1]{RoxanasPhD2017}. The result in \cite[Lemma 2.2]{Chang1991} proves that if the solution $\mathbf{u}$ to \eqref{HMHFeq} is of the form \eqref{polcoord}, then $u$ solves \eqref{PDE}.   
Note that due to the structure  of the solution $\bu$ in \eqref{polcoord} the unit length constraint is satisfied. The
partial differential equation in \eqref{PDE2} has a strong nonlinearity due to the term  $\sin(2u)$ and has a singular behavior for $x \downarrow 0$. 
  The Dirichlet energy in terms of the  function $\mathbf{u}$ is given by 
  \begin{align*}
  	E(\mathbf{u}) = \frac{1}{2} \int_\Omega | \nabla \mathbf{u} |^2 \,\mathrm{d} \mathbf{x},
  \end{align*}
 which by transformation into polar coordinates and using radial symmetry yields
 \begin{equation}
 	E(\mathbf{u}) =  \pi \int_0^1 \big( u_x^2 + \frac{\sin^2 u}{x^2})x \, dx =: \pi \cE(u). \label{energy}
 \end{equation} 
Equation \eqref{PDE2} has been used in models of nematic liquid crystals \cite{vandHout2001}. More importantly, 
 the radially symmetric case \eqref{PDE2} plays a fundamental role in the analysis of HMHF.  
 In the seminal works \cite{ChangDingYe1992,Chang1991} the authors study  finite time singularities of \eqref{HMHFeq}
 for the two-dimensional case $N=2$. In these studies the radially symmetric case  \eqref{PDE2} 
 plays a crucial role. It is shown that for this problem with initial data $u_0$ with $|u_0| \leq \pi$ a unique global smooth solution exists, whereas for the case with  $|u_0(1)|> \pi$ the solution blows up in finite time. The latter means that for the solution $u$ the derivative at $x=0$ becomes arbitrary large: $\lim_{ x\downarrow 0} |u_x(x,t)| \to \infty$ for $t \uparrow T_{\rm crit}$. In these analyses of the problem \eqref{PDE2} the maximum principle is an important tool.
The work \cite{ChangDingYe1992} has motivated further investigations of the blow-up behavior, e.g., \cite{Bertsch_DalPasso_vanderHout_2002,Topp02} where infinitely many solutions of \eqref{PDE2} are constructed whose energy is  bounded by the initial energy for all times $t$, but can increase at some $t>0$, even for a smooth initial condition. 
A further related topic is the analysis   of the blow up rate. Several theoretical aspects of blow up rates of solutions of \eqref{PDE2} are studied in  \cite{vdBergHulshofKing2003,AngHulsMatano09,RaphaelSchweyer13}.
In a series of works, Gustafson \emph{et al.} \cite{GUAN20091,Gustafson_Nakanishi_Tsai_2010} investigate the well-posedness and regularity of the $m$-equivariant version of \eqref{PDE2}. Hocquet \cite{Hocquet19} has studied well-posedness and regularity of a stochastic version of \eqref{PDE2}.

 There are only very few papers in which numerical aspects  of \eqref{PDE2} are treated. In \cite{Schans06} an adaptive method is presented, based on finite element discretization in space combined with a stiff ODE solver. The paper \cite{HaynesHuangZegeling2013} treats a moving mesh ansatz for the discretization of \eqref{PDE2}, using  finite differences  in space in combination with an ODE solver for stiff systems. In both works, the authors are interested in capturing the blow up behavior of the solution with their numerical method. 

We are not aware of any literature in which for a discretization of \eqref{PDE2}, even in the regime of smooth solutions, a rigorous error analysis is presented. Such an error analysis is the main contribution of this paper. We consider a very basic discretization of \eqref{PDE2}, namely a second order finite difference discretization in space  combined with a semi-implicit Euler method in time. The semi-implicit Euler method results in a linear problem in each time step. We restrict to the regime of smooth solutions of the continuous problem and present an analysis based on the classical ``stability plus consistency'' framework. Note, however, that in the analysis of the consistency error we also have to bound the linearization error.  The  analyses of blow up behaviour of the continuous problem \eqref{PDE2} show that there is a critical dependence of stability properties of the problem on  the size of the initial condition  ($|u_0| \leq \pi$: globally smooth solution; $|u_0(1)| >\pi$: finite time blow up).   In view of this, it is not surprising that the key difficulty in the error analysis of the discrete scheme lies in the derivation of satisfactory stability estimates. Our discrete stability analysis uses $M$-matrix theory, which in a certain sense mimics the maximum principle arguments used in the analysis of the continuous problem, cf. \cite{Chang1991}. A key special ingredient in our stability analysis is that besides stability estimates $\|\bu^n\|_\infty \leq \|\bu^0\|_\infty$, where $\bu^n$ denotes the vector of discrete approximations after $n$ semi-implicit Euler time steps, we also derive estimates of the form $ \|D^{-\alpha}\bu^n\|_\infty \leq \|D^{-\alpha}\bu^0\|_\infty$, with $D$ a diagonal scaling matrix, cf. Section~\ref{sec:Stability}. The latter estimates are a discrete analogon of the control of the derivative of the continuous solution at $x=0$. Based on these stability results  we derive a discretization error bound that is of optimal order with respect to both the time step and the spatial mesh size, apart from a logarithmic term in the spatial mesh size. We also present discrete energy estimates that mimic the decrease of the energy \eqref{energy} of the continuous solution.   

The remainder of the paper is organized  as follows.  In Section~\ref{sectdiscrete} we formulate the continuous problem and its discretization. For this discrete problem a   stability analysis is presented in Section~\ref{sec:Stability}. In Section~\ref{Secdiscrerror} the stability estimates are combined with an analysis of consistency (including linearization) errors, resulting in an (almost) optimal discretization error bound, cf. Theorem \ref{mainthm}. In Section~\ref{sec:Energy} discrete energy dissipation estimates are derived. We validate our theoretical findings in Section~\ref{sec:NumericalExp} with numerical examples which demonstrate convergence rates, stability properties and discrete energy dissipation.

\section{A radially symmetric HMHF problem and its discretization} \label{sectdiscrete}
Let $I:=[0,1]$ and given $u_0: I \to \R$ with $u_0(0)=u_0(1)=0$, we are interested in solutions $u: I \times [0,T] \to \R$,  ($u=u(x,t)$) of 
 \begin{equation} \label{PDE} \begin{split}
  u_t & = u_{xx} + \frac{1}{x}u_x - \frac{\sin(2u)}{2 x^2} \quad \text{on}~ I \times [0,T]  \\
   u(\cdot, 0) & = u_0 \quad \text{on}~I, \quad u(0,t)=u(1,t)=0 \quad \text{for}~t \in [0,T].
  \end{split} \end{equation}
  \begin{remark} \rm
   In analyses of HMHF problems the  case with an \emph{in}homogeneous boundary condition $u_0(1)=b \neq 0$ is often treated. To simplify the presentation we restrict to the case of a homogeneous boundary condition $u_0(1)=0$. With rather straightforward modifications our analysis can also be applied to the inhomogeneous case. In a finite difference discretization of the inhomogeneous boundary condition, the boundary data is shifted to the right hand side and the remaining discrete operator corresponds to a homogeneous boundary condition, to which our error analysis can be applied.  
  \end{remark}

  The linear part of the spatial differential operator in \eqref{PDE} is denoted by $Lu= - u_{xx}-\frac{1}{x}u_x$.
  The $L^2$ scalar product on $I$ is denoted by $(\cdot,\cdot)$, and  we use a weighted $L^2$-scalar product $(f,g)_w:=\int_0^1 x f(x)g(x)\, dx$.  Note that for  $v\in H^1_0(I)$: 
\begin{equation} \label{L2weighted}
 (Lu,v)_w=(-u_{xx} - \frac{1}{x}u_x,v)_w=(-x u_{xx} - u_x,v)=(xu_x,v_x)=(u_x,v_x)_w,
\end{equation}
hence $L$ is symmetric positive definite with respect to $(\cdot,\cdot)_w$. 
We introduce a basic discretization of the problem \eqref{PDE}. First we explain the discretization in space, which uses standard finite differences. We use a uniform grid with mesh size $h$ such that $(N+1)h=1$, $N \in \mathbb{N}$ and mesh points $x_i=ih$, $0 \leq i \leq N+1$. The second derivative $-\partial_{xx}u(x_i)$ is approximated with the finite difference $\tfrac{1}{h^2}\left(-u(x_{i-1}) + 2u(x_i) -u(x_{i+1})\right)$ and the first derivative $\partial_{x}u(x_i)$ by the central difference $\tfrac{1}{2h}\left(-u(x_{i-1}+u(x_{i+1})\right)$. We introduce the $N\times N$ matrices
\[
  A:=\frac{1}{h^2}{\rm tridiag}(-1,2,-1), \quad B:=\frac{1}{2h}{\rm tridiag}(-1,0,1),\quad D:={\rm diag}(x_1, \ldots,x_N),
\]
and $C:=A-D^{-1}B$. Note that $C$ represents a discretization of the linear differential operator $L$. 
For the nonlinear term we introduce, for given $\bu=(u_1, \ldots, u_N)^T\in \R^N$ the diagonal matrix 
\[
  G(\bu):={\rm diag}\Big( \frac{\sin (2u_1)}{2u_1}, \ldots, \frac{\sin (2u_N)}{2u_N}\Big). 
\]
For discretization in time of  we use a fixed time step $\Delta t=\frac{T}{M}$, $M \in \mathbb{N}$. The nonlinear term in  \eqref{PDE} can be written as $\frac{\sin(2u)}{2 x^2}=\frac{\sin(2u)}{2 u}\frac{u}{x^2}$. For the first factor, which is nonlinear, we have  $\left|\frac{\sin(2u)}{2 u}\right| \in [0,1]$, and the second factor, which is linear, may blow up for $x \downarrow 0$. Due to the parabolic nature of the PDE it is natural to use an implicit time stepping scheme and the two factor splitting of the nonlinear term suggests an semi-implicit  linearization.   
This motivates the following basic discretization method for \eqref{PDE}:
\begin{equation} \label{discrete}
 \frac{\bu^{n+1}-\bu^n}{\Delta t}= - C \bu^{n+1}-G(\bu^n)D^{-2}\bu^{n+1}, \quad 0 \leq n \leq M-1,
\end{equation}
with $\bu^0:=(u_0(x_1), \ldots u_0(x_N))^T$. Note that we use a simple first order linearization for the nonlinear term in \eqref{PDE}.
Well-posedness of this scheme will be discussed below.
\section{Stability estimates} \label{sec:Stability}
A key point in the analysis are discrete stability bounds, not only for the approximation $\bu^n$ of $u(\cdot,t_n)$ but also for $D^{-\alpha}\bu^n$ with $\alpha \in (0,1]$, which are approximations of $x^{-\alpha}u(\cdot, t_n)$.
\begin{remark} \label{Rem1}
 \rm In a certain sense boundedness of $\|D^{-\alpha}\bu^n\|_\infty$, uniformly in $h$ and $n$, mimics the ``smoothness'' of the solution $u(0,t_n)$ of the continuous problem. This smoothness in particular means that $|u_x(0,t)|$ is bounded (``the derivative at 0 does not blow up''). Using $u(0,t)=0$ we have $u_x(0,t)= \frac{1}{x}u(x,t)+ \mathcal{O}(x)$ ($ x \to 0$). Hence boundednes of $|\frac{1}{x}u(x,t)|$ close to 0 is equivalent to boundedness of $|u_x(0,t)|$. The discrete analogon of this is the boundedness of $\|D^{-1}\bu^n\|_\infty$. For 
 $\alpha \in (0,1)$ we have a discrete analogon of Hölder continuity (with exponent $\alpha$) at $x=0$. 
\end{remark}

The stability analysis is based on $M$-matrix theory. We recall a few basic results from that theory. Define 
\[
  Z_N:=\left\{ \, A \in \R^{N \times N}~|~ a_{ij} \leq 0 \quad \text{for all}~i\neq j\,\right\}.
\]
A matrix $A$ is called an $M$-matrix if $A \in Z_n$ and $\det(A) \neq 0 $, $A^{-1} \geq 0 $ (inequality elementwise).
A matrix $A$ is called  weakly diagonally dominant if $\sum_{j \neq i} |a_{ij}| \leq |a_{ii}|$ for all $i$. We use the notation $\bfone:=(1,\ldots,1)^T  \in \R^N$. Furthermore, for $\bx \in \R^N$, $|\bx|:=(|x_1|, \ldots, |x_N|)^T$. From the literature we have the following result ({Fiedler Thm. 5.1).
\begin{lemma} \label{lemma1}
A matrix $A \in Z_N$ is an $M$-matrix if and only if there exists $\bx > 0$ such that $A \bx >0$.
\end{lemma}
\ \\
Using this we obtain the following.
\begin{lemma} \label{lemma2}
 If $A \in Z_N$ is weakly diagonally dominant and $a_{ii} \geq 0$ for all $i$, then for all $\delta \geq 0$ we have $\det(I+\delta A) \neq 0$ and
 \[
   \|(I+ \delta A)^{-1}\|_\infty \leq 1.
 \]
\end{lemma}
\begin{proof} If $A \in Z_N$, then $B:=I+\delta A \in Z_N$. From the weak diagonal dominance and $a_{ii} \geq 0$ we get $1+\delta \sum_{j=1}^N a_{ij} \geq 1$ for all $i$. Hence, $B\bfone \geq \bfone$ holds. From Lemma~\ref{lemma1} we conclude that $B$ is an $M$-matrix, hence $B$ invertible and $B^{-1} \geq 0$, which implies that $\|B^{-1}\|_\infty =\max_{1 \leq i \leq N} \sum_{j=1}^N\left| (B^{-1})_{ij}\right|=\max_{1 \leq i \leq N} \sum_{j=1}^N (B^{-1})_{ij} =\|B^{-1} \bfone \|_\infty$ holds. Finally note that $0 \leq   B^{-1} \bfone \leq B^{-1} B\bfone = \bfone$ holds.  
\end{proof}
\ \\
Using this result a stability analysis of the scheme \eqref{discrete} is very straightforward. Note that for the matrix $C$ we have
\begin{align*}
  -C &= {\rm tridiag}(\gamma_i,\delta_i,\beta_i)_{1 \leq i \leq N} \quad \text{with} \\
  \gamma_i &= \tfrac{1}{h^2}(1-\tfrac{1}{2i}), \quad 2 \leq i \leq N, \\
  \delta_i &= - \tfrac{2}{h^2}, \quad 1 \leq i \leq N, \\ 
  \beta_i & = \tfrac{1}{h^2} (1 +\tfrac{1}{2i}), \quad 1 \leq i \leq N-1.
  \end{align*}
\begin{lemma} \label{lemma3}
 Let $\bu^0$ be such that $\|\bu^0\|_\infty \leq \frac{\pi}{2}$ holds. Then the scheme \eqref{discrete} is well-defined  and $\|\bu^n\|_\infty \leq \|\bu^0\|_\infty$ for all $n$ holds. 
\end{lemma}
\begin{proof} We use induction. Assumme $\|\bu^n\|_\infty \leq \frac{\pi}{2}$ holds. The matrix $C$ is weakly diagonally dominant, $c_{ii} \geq 0$ for all $i$ and $C \in Z_N$. The matrix $G(\bu^n)D^{-2}$ is diagonal with positive diagonal entries. It follows that the matrix $M:= C +G(\bu^n)D^{-2}$ is weakly diagonally dominant, $m_{ii} \geq 0$ for all $i$ and $M \in Z_N$. From Lemma~\ref{lemma2} it follows that $I+\Delta t M$ is invertible and $\|(I+ \Delta t M)^{-1}\|_\infty \leq 1$ holds. Thus in \eqref{discrete} we have
\[
 \|\bu^{n+1}\|_{\infty} = \|(I+ \Delta t M)^{-1}\bu^n\|_\infty \leq \|\bu^n\|_\infty \leq \frac{\pi}{2}.
\]
\end{proof}
\ \\
\begin{remark} \rm We need the condition $\|\bu^0\|_\infty \leq \frac{\pi}{2}$, whereas for the continuous problem, for existence of a smooth solution one needs $\|u_0\|_{L^\infty(I)} \leq \pi$.  We were not able to bridge the gap between $\frac{\pi}{2}$  and $\pi$. 
\end{remark}
\begin{assumption} \label{assump:Initial_cond} \rm In the remainder we assume that $\|\bu^0\|_\infty \leq \frac{\pi}{2}$ holds.
\end{assumption}

Clearly, for the discretization error analysis we need to control the nonlinearity in the term $G(\bu^n)$. Note that in \eqref{discrete} this term is multiplied by $D^{-2}$ (with $\|D^{-2}\|_\infty= h^{-2}$). It turns out, that for satisfactory error bounds we need control of $\|D^{-\alpha}\bu^n\|_\infty$, with $\alpha > \tfrac12$. 

For the  control of $\|D^{-\alpha}\bu^n\|_\infty$ we use a stability estimate for the matrix $D^{-\alpha}C D^\alpha$. For this we first note the following. The matrix $C$ is a discretization of the differential operator $Lu:=-u_{xx}-\tfrac{1}{x}u_x$. As simple computation yields
\begin{equation} \label{s1}
   x^{2-\alpha}L(x^\alpha)= - \alpha^2 \quad \text{for all}~ \alpha \in [0,1].
\end{equation}
A discrete analogon of this property also holds.
\begin{lemma} \label{lemA}
The following holds:
\[
   D^{2-\alpha}C D^\alpha \bfone \geq - \alpha^2 \bfone \quad \text{for all}~ \alpha \in [0,1].
\]
\end{lemma}
\begin{proof} 
Using $D^\alpha\bfone= (x_1^\alpha, \ldots,x_N^\alpha)^T \geq 0$,  $-C = {\rm tridiag}(\gamma_i,\delta_i,\beta_i)_{1 \leq i \leq N}$ and with $\gamma_1:= \frac{1}{h^2}(1-\frac{1}{2})$, $\beta_N:= \frac{1}{h^2} (1 +\frac{1}{2N})$ we obtain, for $1 \leq i \leq N$:
\begin{equation} \label{E4} \begin{split}
 \big(-D^{2-\alpha}C D^\alpha \bfone\big)_i & \leq x_i^{2-\alpha}\big(\gamma_i x_{i-1}^\alpha +\delta_i x_i^\alpha + \beta_i x_{i+1}^\alpha \big) \\
   & = x_i^{2}\big(\gamma_i (\frac{x_{i-1}}{x_i})^\alpha +\delta_i  + \beta_i (\frac{x_{i+1}}{x_i})^\alpha \big)\\
   & = i^2 \big( (1-\tfrac{1}{2i})(1- \tfrac{1}{i})^\alpha -2 +(1+\tfrac{1}{2i})(1+\tfrac{1}{i})^\alpha\big) =f(\tfrac{1}{i}),
  \end{split}
\end{equation}
with 
\begin{equation} \label{deff} \begin{split}
 f(y) & := \tfrac{1}{y^2}\big( (1- \tfrac12 y)(1-y)^\alpha -2 + (1+\tfrac12 y)(1+y)^\alpha\big) \\
    & = \tfrac{1}{y^2}\big((1+y)^\alpha +(1-y)^\alpha -2 + \tfrac12 y ((1+y)^\alpha -(1-y)^\alpha)\big), \quad y \in (0,1].  
\end{split} \end{equation}
For $|y| < 1$ we have the convergent series
\begin{align*}
  (1+y)^\alpha & = 1+\alpha y + \frac{\alpha(\alpha-1)}{2!} y^2 + \frac{\alpha(\alpha-1)(\alpha-2)}{3!} y^3 + \ldots  \\
    & = \sum_{j=0}^\infty c_j y^j, \quad c_0=1, \quad c_j=\tfrac{1}{j!} \Pi_{m=0}^{j-1} (\alpha -m) \quad j \geq 1.
\end{align*}
This yields
\[ \begin{split}
 (1+y)^\alpha +(1-y)^\alpha -2  & = 2 \sum_{j=1}^\infty c_{2j} y^{2j} = 2 \sum_{j=0}^\infty c_{2j+2} y^{2j+2} \\
 \tfrac12 y \big((1+y)^\alpha -(1-y)^\alpha\big) &= \sum_{j=0}^\infty c_{2j+1} y^{2j+2}.
\end{split}
\]
Using this in \eqref{deff} we obtain
\[
  f(y)= \frac{1}{y^2}\left( 2 \sum_{j=0}^\infty c_{2j+2} y^{2j+2} +  \sum_{j=0}^\infty c_{2j+1} y^{2j+2} \right) =
   \alpha^2 + \sum_{j=1}^\infty (2 c_{2j+2} + c_{2j+1}) y^{2j}. 
\]
Now note, using $\alpha \in [0,1]$, 
\[
  2 c_{2j+2} + c_{2j+1}= \frac{1}{(2j+1)!}  \frac{\alpha-j}{j+1}\Pi_{m=0}^{2j} (\alpha -m)   \leq 0 \quad \text{for all}~j \geq 1,
\]
which implies $f(y) \leq \alpha^2$ for all $y \in (0,1)$. Due to continuity this also holds for $y=1$. Using this in \eqref{E4} completes the proof. 
 \end{proof}
\ \\[1ex]
We introduce the notation $\bu_\alpha^n:=D^{-\alpha} \bu^n$.
From \eqref{discrete} we obtain
\begin{equation} \label{inter}
 (I+\Delta t (D^{-\alpha}CD^\alpha+ G(\bu^n)D^{-2}))\bu_\alpha^{n+1}=  \bu_\alpha^n, \quad n \geq 0, ~\bu_\alpha^0=D^{-\alpha}\bu^0.
\end{equation}
The matrix $D^{-\alpha}CD^\alpha$ is an $M$-matrix, but \emph{not} weakly diagonally dominant. We use the property of Lemma~\ref{lemA} to derive a stability result for  $(\bu_\alpha^n)_{n \geq 0}$.
\begin{lemma} \label{lemma4}
 Take $\alpha \in [0,1)$. Let $c_\alpha \in (0,\tfrac{\pi}{2}]$ be such that $\frac{\sin(2c_\alpha)}{2c_\alpha}=\alpha^2$.  Assume $\bu^0$ satisfies $\|\bu^0\|_\infty \leq c_\alpha$. Then the following holds:
 \[
   \|\bu_\alpha^n\|_\infty \leq \|\bu_\alpha^0\|_\infty \quad \text{for all}~n \geq 0.
 \]
\end{lemma}
\begin{proof}
 We use similar arguments as in the proof of Lemma~\ref{lemma3}. Take $\bu^0$ such that $\|\bu^0\|_\infty \leq c_\alpha$. From Lemma~\ref{lemma3} we then have $\|\bu^n\|_\infty \leq c_\alpha$ for all $n$. Since $\frac{\sin (2y)}{2y} \geq \alpha^2$ for all $y \in [-c_\alpha,c_\alpha]$ it follows that
 $G(\bu^n) \geq \alpha^2 I$. Hence, $M_\alpha:=D^{-\alpha}CD^\alpha+ G(\bu^n)D^{-2} \geq D^{-\alpha}CD^\alpha +\alpha^2 D^{-2}$. Note that $M_\alpha \in Z_N$ and $(M_\alpha)_{ii} \geq  0$ holds. Using Lemma~\ref{lemA} we get $M_\alpha \bfone \geq D^{-2}( D^{2-\alpha}C D^\alpha \bfone + \alpha^2 \bfone) \geq 0$, and thus $M_\alpha$ is weakly diagonally dominant.   From Lemma~\ref{lemma2} it follows that $\|(I+ \Delta t M_\alpha)^{-1}\|_\infty \leq 1$ holds. Thus in \eqref{inter} we have
\[
 \|\bu_\alpha^{n+1}\|_{\infty} = \|(I+ \Delta t M_\alpha)^{-1}\bu_\alpha^n\|_\infty \leq \| \bu_\alpha^n\|_\infty \leq  \|\bu_\alpha^0\|_\infty.
\]
\end{proof}
\ \\[1ex]
Note that the result above can not be extended to $\alpha=1$. 

Besides the maximum norm we use two other vector norms in the error analysis below. The Euclidean scalar product on $\R^N$ is denoted by $\langle \cdot,\cdot\rangle$ and the corresponding scaled norm by $\|\bv\|_{2,h}^2:=h \langle \bv,\bv\rangle = h\sum_{i=1}^N v_i^2$. We also use the diagonally scaled Euclidean norm $\|\bv\|_{D,h}:=\|D^\frac12 \bv\|_{2,h}$, i.e., $\|\bv\|_{D,h}^2= h \langle D\bv, \bv\rangle$. The matrix norm corresponding  to $\|\cdot\|_{2,h}$ is denoted by $\|\cdot\|_2$. Note that the diagonally weighted scalar product is natural in  the sense that $C$ is symmetric with respect to this scalar product. It is the discrete analogon of the weighted $L^2$ scalar product $(\cdot,\cdot)_w$ used in \eqref{L2weighted}.

One easily checks that the matrix $DC$ is symmetric, which is the discrete analogon of the symmetry, in the $L^2$ scalar product of $u \to -xu_{xx}-u_x$. We finally derive a stability result in the Euclidean norm. 
 \begin{lemma} \label{lemma5}
  The following holds
  \begin{equation} \label{bound1}
    \|\big(I+ \Delta t(D^\frac12 C D^{- \frac12} +G(\bu^n)D^{-2})\big)^{-1}\|_2 \leq 1.
  \end{equation}
 \end{lemma}
 \begin{proof}
The matrix  $C+G(\bu^n)D^{-2}$ is an $M$-matrix, cf. proof of Lemma~\ref{lemma3}. From this and the symmetry of $DC$ we get that $D\big(C+G(\bu^n)D^{-2})=DC+G(\bu^n)D^{-1}$ is a symmetric $M$-matrix. Thus this matrix is positive definite, cf. 
\cite[Theorem 5.1]{Fiedler}. 
Hence $D^{-\frac12}\big(DC+G(\bu^n)D^{-1}\big)D^{-\frac12}= D^\frac12 C D^{- \frac12} +G(\bu^n)D^{-2}$ is symmetric positive definite.  From this it follows that $\min_{\bx \in \R^N, \bx \neq 0}\frac{\langle (I+ \Delta t(D^\frac12 C D^{- \frac12} +G(\bu^n)D^{-2})) \bx,\bx\rangle}{\langle \bx, \bx \rangle} \geq 1$  holds, which implies the estimate \eqref{bound1}. 
\end{proof}

\section{Discretization error bounds} \label{Secdiscrerror}
For linearization we use the following elementary estimate.
\begin{lemma} \label{lemma8}
 Define $g(y):=\frac{\sin(2y)}{2y}$, $y \neq 0$, $g(0)=1$. The following holds:
 \[ 
  |g(y) - g(z)| \leq \tfrac43 \max \{|y|,|z|\} |y-z| \quad \text{for all}~ y,z \in \R.
 \]
\end{lemma}
\begin{proof}
Elementary computation yields $|g'(y)| \leq \tfrac43 |y|$ for all $y \in \R$. Thus we obtain
\[ \begin{split}
 |g(y) - g(z)| & = \left| \int_0^1 g'(z+t(y-z)) \, dt (y-z)\right| \leq \tfrac43 \int_0^1 |z+t(y-z)| \, dt |y-z| \\
  & \leq \tfrac43 \int_0^1 (1-t)|z|+ t|y| \, dt |y-z| \leq \tfrac43 \max \{|y|,|z|\} |y-z|.
\end{split} \]
\end{proof}
\ \\[1ex]
For a sufficiently smooth function $v(x,t)$, $x \in I$, $t \in [0,T]$ we introduce a corresponding grid function $\bv(t):=\left(v(x_1,t),\ldots, v(x_N,t)\right)^T \in \R^N$. The grid functions for $v_t(\cdot,t)$, $v_x(\cdot,t)$, $v_{xx}(\cdot,t)$ are denoted by $\bv_t(t)$, $\bv_x(t)$ and $\bv_{xx}(t)$, respectively. Let $u$ be the solution of \eqref{PDE}, which is assumed to be sufficiently smooth. From Taylor expansion we have 
\begin{align*}
 \bu_t(t_{n+1}) & = \frac{\bu(t_{n+1}) - \bu(t_n)}{\Delta t} + \Delta t  \, \be_{\partial t}(t_{n+1}), \quad  \be_{\partial t}(t_{n+1})=\tfrac12 \left(u_{tt}(x_i,\xi_{n+1})\right)_{1 \leq i \leq N}, \\
 \bu_{x}(t_{n+1}) & = B \bu(t_{n+1}) + h^2 \be_{\partial x}(t_{n+1}), \quad \be_{\partial x}(t_{n+1})= - \tfrac16 \left( \frac{\partial^3 u}{\partial x^3}(\xi_i,t_{n+1})\right)_{1 \leq i \leq N}, \\
 \bu_{xx}(t_{n+1}) &= -A \bu(t_{n+1}) + h^2 \be_{\partial^2 x}(t_{n+1}), ~\be_{\partial^2 x}(t_{n+1})= -\tfrac{1}{12}\left(\frac{\partial^4 u}{\partial x^4}(\eta_i,t_{n+1})\right)_{1 \leq i \leq N}.
\end{align*}
Thus we get  
\[
  \bu_t(t_{n+1}) = \bu_{xx}(t_{n+1}) + D^{-1}\bu_{x}(t_{n+1}) - G(\bu(t_{n+1}))D^{-2}\bu(t_{n+1})
\]
and  
\begin{equation} \label{erreq1} \begin{split}
 \frac{\bu(t_{n+1})- \bu(t_n)}{\Delta t}  & = - C \bu(t_{n+1}) - G(\bu(t_{n+1}))D^{-2}\bu(t_{n+1})  \\  & \quad  -\Delta t \,\be_{\partial t}(t_{n+1}) +   h^2 \be_{\partial^2 x}(t_{n+1})+h^2 D^{-1}\be_{\partial x}(t_{n+1}) .
 \end{split}
\end{equation}
The discretization error is denoted by $\be^n:=\bu(t_n)-\bu^n$, with $\bu^n$ the solution  of \eqref{discrete}. We then obtain the following recursion for the discretization error, with $\be^0:=0$:
\begin{equation} \label{erreq2} \begin{split}
 \frac{\be^{n+1}- \be^n}{\Delta t}  & = - C\be^{n+1} - \left(G(\bu(t_{n+1}))D^{-2}\bu(t_{n+1})- G(\bu^n)D^{-2}\bu^{n+1}\right)  \\  & \quad  -\Delta t \,\be_{\partial t}(t_{n+1}) +   h^2 \be_{\partial^2 x}(t_{n+1})+h^2 D^{-1}\be_{\partial x}(t_{n+1}), \quad 0 \leq n \leq M-1.
 \end{split}
\end{equation}
 In \cite{Chang1991} it is shown that if the initial condition satisfies $u_0 \in C^{1+\gamma}(I)$ for some $\gamma\geq 0$ and $\|u_0\|_\infty \leq \pi$, then the solution $u$ of \eqref{PDE} has  regularity $u \in C^\gamma(I\times [0,T)])\cap C^{2+\gamma}(I\times (0,T))$. In the main theorem below, we assume the initial condition $u_0$ to be smooth enough such that the $L^\infty$ norms of derivatives of $u$ that appear are finite. 

\begin{theorem} \label{mainthm}    Take $\alpha  \in (\tfrac12, 1)$ and assume $\|\bu^0\|_\infty \leq c_\alpha$, with $c_\alpha$ as in Lemma~\ref{lemma4}.
Define $d_\alpha:=\|D^{-\alpha}\bu^0\|_\infty$.  
For the error $\be^n$ the following holds:
\begin{equation} \label{optboound}
  \|\be^n\|_{D,h} \leq   c \big( \Delta t + h^2 |\ln h|^\frac12), \quad 1 \leq n \leq M,
\end{equation}
with a constant $c$ that depends only on $\|\frac{\partial^j u}{\partial  x^j}\|_{L^\infty(I\times[0,T])}$,  $0 \leq j \leq 4$, $\| \frac{\partial^j u}{\partial t^j}\|_{L^\infty(I\times[0,T])}$,  $0 \leq j \leq 2$,  $T$, $\alpha$, $c_\alpha$ and $d_\alpha$. The dependence of $c$ on these quantities can be deduced from the proof.
\end{theorem}
\begin{proof} We use $c$ for a (varying) constant that depends only on the quantities mentioned above. 
From \eqref{erreq2} we obtain, after multiplication with $D^\frac12$
\[ \begin{split}
  & (I+\Delta t D^\frac12C D^{-\frac12})D^\frac12\be^{n+1} \\ & =D^{\frac12}\be^n - \Delta t\, E    -(\Delta t)^2 \,D^\frac12 \be_{\partial t}(t_{n+1}) +  \Delta t\,  h^2 D^\frac12\be_{\partial^2 x}(t_{n+1})+ \Delta t\, h^2 D^{-\frac12}\be_{\partial x}(t_{n+1}), \\
  & E:= \left(G(\bu(t_{n+1}))D^{-\frac32}\bu(t_{n+1})- G(\bu^n)D^{-\frac32}\bu^{n+1}\right).
\end{split} \]
For linearization of the term $E$ we use the splitting
\begin{equation} \label{splitE}
 \begin{split}
  E & = G(\bu(t_{n+1}))D^{-\frac32}\bu(t_{n+1})- G(\bu(t_n))D^{-\frac32}\bu(t_{n+1}) \\
     & + G(\bu(t_n))D^{-\frac32}\bu(t_{n+1})- G(\bu^n)D^{-\frac32}\bu(t_{n+1}) \\
     & + G(\bu^n)D^{-\frac32}\bu(t_{n+1})- G(\bu^n)D^{-\frac32}\bu^{n+1} =:E_1+E_2+E_3.
 \end{split}
\end{equation}
The term $E_3$ is shifted to the left hand side. We thus get  
\begin{equation} \label{erreq8} \begin{split}
  & \big(I+\Delta t (D^\frac12C D^{-\frac12} +G(\bu^n)D^{-2})\big)D^{\frac12}\be^{n+1} \\ & =D^{\frac12}\be^n - \Delta t\, (E_1+E_2)  \\ &   -(\Delta t)^2 \,D^\frac12 \be_{\partial t}(t_{n+1}) +  \Delta t\,  h^2 D^\frac12\be_{\partial^2 x}(t_{n+1})+ \Delta t\, h^2 D^{-\frac12}\be_{\partial x}(t_{n+1}).
\end{split} \end{equation}
We  consider the term $E_1$. Note that the $\|\cdot\|_{2,h}$ matrix norm is scaling invariant. Replacing it by $\|\cdot\|_2$ we obtain
\[
 \|E_1\|_{2,h} \leq \|\big(G(\bu(t_{n+1}))- G(\bu(t_n))\big)D^{-\frac12}\|_2 \|D^{-1}\bu(t_{n+1})\|_{2,h}.
\]
Note that $\|\bv\|_{D,h}\leq \|\bv\|_{2,h} \leq \|\bv\|_\infty$ holds for all $\bv \in \R^N$. 
For the continuous solution $\bu(t_n)$ we have
\begin{equation} \label{erreq4} \begin{split}
  \|D^{-1}\bu(t_n)\|_{\infty} & = \max_{1 \leq i \leq N} \left| \frac{u(x_i,t_n)}{x_i}\right| \\ & =\max_{1 \leq i \leq N} \frac{1}{x_i} \left|
  \int_0^{x_i} u_x(s,t_n)\, ds\right| \leq \|\frac{\partial u}{\partial  x}\|_{L^\infty(I\times[0,T])}=c_1.
\end{split} \end{equation}
Using $x_i^{-\frac12} \leq x_i^{-1}$,  Lemma~\ref{lemma8} and  \eqref{erreq4} we obtain
\[ \begin{split}
 & \|\big(G(\bu(t_{n+1})) - G(\bu(t_{n}))\big)D^{-\frac12}\|_2  \leq  \max_{1 \leq i \leq N} \frac{1}{x_i} |g(u(x_i,t_{n+1}))-g(u(x_i,t_{n}))| \\ &  \leq \tfrac43 c_1 \max_{1 \leq i \leq N} |u(x_i,t_{n+1})-u(x_i,t_n)| \\ & \leq \tfrac43 c_1 \Delta t \|u_t\|_{L^\infty(I \times [0,T])} = c \Delta t.
\end{split}\]
Using this and \eqref{erreq4} yields
\begin{equation} \label{erreq5}
 \|E_1\|_{2,h} \leq c\, \Delta t.
\end{equation}
For the term $E_2$ we note, for $1 \leq i \leq N$,
\[ \begin{split}
 |(E_2)_i| & =\left| \big(g(u(x_i,t_{n})) - g(u_i^n)\big)\big)x_i^{-\frac32} u(x_i,t_{n+1})\right| \\
 & \leq \tfrac43 \max \{|u(x_i,t_{n})|,|u_i^n|\} |e_i^n| x_i^{-\frac32} |u(x_i,t_{n+1})| 
  \\ &  =\tfrac43 \max \{|u(x_i,t_{n})|,|u_i^n|\} x_i^{-\alpha} |e_i^n| x_i^{\alpha-\frac32} |u(x_i,t_{n+1})| 
\end{split}\]
Using $x_i^{-\alpha}|u(x_i,t_{n})| \leq x_i^{-1}|u(x_i,t_{n})| \leq  c_1$ and 
\begin{equation} \label{newlabel} x_i^{-\alpha} |u_i^n| \leq \|D^{-\alpha} \bu^n\|_\infty \leq \|D^{-\alpha} \bu^0\|_\infty= d_\alpha,
\end{equation}
cf. Lemma~\ref{lemma4},  this yields
\[
 |(E_2)_i| \leq \tfrac43 \max\{c_1, d_\alpha\} |x_i^\frac12 e_i^n| |x_i^{\alpha-2}u(x_i,t_{n+1})|.
\]
With Cauchy-Schwarz we thus get
\begin{equation} \label{estE2}
 \|E_2\|_{2,h} \leq  \tfrac43 \max\{c_1, d_\alpha\} \|\be^n\|_{D,h} \left(h \sum_{i=1}^N x_i^{2\alpha-4} u(x_i,t_{n+1})^2 \right)^\frac12.
\end{equation}
Using $ \max_{i}|\frac{u(x_i,t_{n+1})}{x_i}| \leq c_1$ and $x_i=ih$  we obtain for the expression above
\begin{align*} 
  & \left(h \sum_{i=1}^N x_i^{2\alpha-4} u(x_i,t_{n+1})^2 \right)^\frac12 \leq c_1 \left(h \sum_{i=1}^N x_i^{2(\alpha-1)}  \right)^\frac12=c_1 \left(h^{2 \alpha -1} \sum_{i=1}^N i^{2(\alpha-1)}  \right)^\frac12 \nonumber \\
 & \leq c_1 \left(h^{2 \alpha -1}\big( 1+ \int_1^{N} z^{2 (\alpha-1)}\, dz \big) \right)^\frac12 \leq \frac{c_1}{\sqrt{2 \alpha -1}}.
\end{align*}
Summarizin*g we obtain
\begin{equation} \label{R5}
   \|E_2\|_{2,h} \leq c \|\be^n\|_{D,h}, \quad \text{with}~ c= \tfrac43 \max\{c_1, d_\alpha\}\frac{c_1}{\sqrt{2 \alpha -1}}.
\end{equation}
Using  the results \eqref{erreq5} and \eqref{R5} in \eqref{erreq8} and with the stability estimate of Lemma~\ref{lemma5} we thus get  
\begin{equation}\label{eq7} \begin{split}
 & \|\be^{n+1}\|_{D,h} \leq (1 + c \Delta t)\|\be^n\|_{D,h} + c\, (\Delta t)^2 \\
  & + (\Delta t)^2 \|D^\frac12 \be_{\partial t}(t_{n+1})\|_{2,h} +  \Delta t\, h^2\| D^\frac12\be_{\partial^2 x}(t_{n+1})\|_{2,h}+ \Delta t\, h^2 \|D^{-\frac12}\be_{\partial x}(t_{n+1})\|_{2,h}.
\end{split} \end{equation}
We finally estimate the three terms in the second line of \eqref{eq7}. We have
\[ \|D^\frac12 \be_{\partial t}(t_{n+1})\|_{2,h} \leq \|\be_{\partial t}(t_{n+1})\|_\infty \leq \frac12 \| \frac{\partial^2 u}{\partial t^2}\|_{L^\infty(I\times[0,T])},
\]
and  
\[
 \| D^\frac12\be_{\partial^2 x}(t_{n+1})\|_{2,h} \leq \| \be_{\partial^2 x}(t_{n+1})\|_\infty \leq \frac{1}{12}\|\frac{\partial^4 u}{\partial  x^4}\|_{L^\infty(I\times[0,T])}.
\]
In  the third term there is a scaling with $D^{-\frac12}$. For this term we obtain, using $\sum_{i=1}^N \frac{1}{i} \leq 1 +\int_1^N \frac{1}{x} \, dx = 1 + \ln N \leq 1 + |\ln h| \leq c |\ln h|$:
\[ \begin{split} 
  \|D^{-\frac12}\be_{\partial x}(t_{n+1})\|_{2,h} & = \frac16 \left( h \sum_{i=1}^N \frac{1}{x_i} \frac{\partial^3 u}{\partial x^3}(\xi_i,t_{n+1})^2 \right)^\frac12 
   \leq \frac{1}{6}\|\frac{\partial^3 u}{\partial  x^3}\|_{L^\infty(I\times[0,T])} \left( \sum_{i=1}^N \frac{1}{i}\right)^\frac12 \\
  & \leq c \|\frac{\partial^3 u}{\partial  x^3}\|_{L^\infty(I\times[0,T])} |\ln h|^\frac12.
\end{split} \]
Using these estimates in \eqref{eq7} we obtain
\[
  \|\be^{n+1}\|_{D,h} \leq (1 + c \Delta t)\|\be^n\|_{D,h} + c\, \Delta t \big( \Delta t + h^2 |\ln h|^\frac12\big).
\]
A standard recursive argument, using $(1+ c \Delta t)^M \leq e^{c T}$ and $\be^0=0$ completes the proof.
\end{proof}

We comment on the result presented in Theorem~\ref{mainthm}. First note that  $d_\alpha=\|D^{-\alpha}\bu^0\|_\infty= \max_{1 \leq i \leq N} x_i^{-\alpha}|u_0(x_i)|$ essentially measures the smoothness of $u_0$ close to 0, cf. Remark~\ref{Rem1}. 
Furthermore, if $u_0 \in C^1(I)$ then $d_\alpha$ is uniformly bounded in $\alpha$: 
\begin{equation} \label{destimate} d_\alpha \leq \|(u_0)_x\|_{L^\infty(I)} \quad \text{for all}~\alpha \in [0,1]. 
\end{equation}
The bound in \eqref{optboound} is optimal with respect to the order of convergence in $\Delta t$ and in $h$, apart from the factor $|\ln h|^\frac12$. To derive this bound it is essential (in our analysis) to take $\alpha > \tfrac12$. For an explanation of this, consider the error propagation relation \eqref{erreq8}. For  deriving the estimates for the local truncation errors in the last line of \eqref{erreq8} we do \emph{not} need $\alpha > \tfrac12$ to hold. The remaining two terms $E_1$, $E_2$ in \eqref{erreq8} are part of the linearization error. For  deriving the bound \eqref{erreq5} for the term $E_1$ we do \emph{not} need $\alpha > \tfrac12$ to hold. To be able to control the term $E_2$, however, we do need the condition $\alpha > \tfrac12$ to be satisfied, cf.~\eqref{R5}. With respect to the choice of $\alpha \in (\frac12,1)$ we note the following. The constant $c$ in \eqref{optboound} becomes smaller if we take a larger $\alpha$, cf. \eqref{R5}. The function $\alpha \to c_\alpha$ is monotonically decreasing on $(\frac12,1)$ with $\lim_{\alpha \to 1}c_\alpha=0$  and thus the condition $\|\bu^0\|_\infty \leq c_\alpha$ becomes more severe for increasing $\alpha$.

\section{Discrete energy estimates} \label{sec:Energy}
Recall that the Dirichlet energy of the HMHF problem in terms of $u$ is given by, cf. \eqref{energy} and \eqref{L2weighted},
\[
  \cE(u)=(Lu,u)_w+\big(\frac{\sin u}{x}, \sin u\big) = (xLu,u)+\big(\frac{\sin u}{x}, \sin u\big). 
\]
The matrix $C$ is the finite difference approximation of the differential operator $L$. The matrix $DC$ is symmetric positive definite, which is the discrete analogon of the fact that $L$ is symmetric positive definite w.r.t $(\cdot, \cdot)_w$, cf.~ \eqref{L2weighted}. We introduce the \emph{discrete Dirichlet energy} functional
\begin{equation}
	\label{discrete_energy}
	\begin{split} \cE_h(\bu) & := h \langle DC \bu,\bu \rangle + h \langle D^{-1} F(\bu),F(\bu) \rangle   
	 \\
\text{with}~~
	F(\bu) & :=  \left( \sin (u_1), \sin(u_2),\hdots,\sin(u_N) \right)^T. \end{split}
\end{equation}
In this section we analyze energy dissipation of the discrete scheme \eqref{discrete} with respect to $\cE_h(\bu)$. 
For this it is convenient to  rewrite \eqref{discrete} as
\begin{equation}
	\label{discrete_v2}
	\frac{\bu^{n+1}-\bu^n}{\Delta t}= - C \bu^{n+1}-G(\bu^{n+1})D^{-2}\bu^{n+1}+\left( G(\bu^{n+1})-G(\bu^{n})\right) D^{-2}\bu^{n+1}.
\end{equation}
\begin{lemma}
	\label{energy_lemma} Take $\alpha  \in [0, 1)$ and assume $\|\bu^0\|_\infty \leq \min\{ 
	\frac{\pi}{4},c_\alpha\}$, with $c_\alpha$ as in Lemma~\ref{lemma4}.
Define $d_\alpha:=\|D^{-\alpha}\bu^0\|_\infty$. 
	 Let $(\bu^{n})_{0 \leq n \leq M}$ solve \eqref{discrete}. The follwing holds for $0 \leq n \leq M-1$, and with $c$ the constant in the estimate \eqref{optboound}:
	\begin{align}   
		\cE_h(\bu^{n+1}) & \leq \cE_h(\bu^{n}), \quad \text{if}~~\Delta t \leq \tfrac34 d_\alpha^{-2} h^{2(1-\alpha)}~\text{holds.} \label{diss2} \\
		\cE_h(\bu^{n+1}) & \leq \cE_h(\bu^{n}) + \tfrac{64}{3} c^2 d_\alpha^2 \big(h^{2(\alpha-1)}(\Delta t)^2 +h^{2(\alpha +1)}|\ln h|\big), \quad \text{if}~~\alpha > \tfrac12~\text{holds.}  \label{diss1}
	\end{align}
\end{lemma}
\begin{proof} 
	We take the $h$-weighted scalar product of \eqref{discrete_v2} with $D(\bu^{n+1}-\bu^n)$. This yields
	\[  \begin{split}
	& \frac{h}{\Delta t} \langle D(\bu^{n+1}-\bu^n),\bu^{n+1}-\bu^n \rangle + h \langle DC \bu^{n+1},\bu^{n+1}-\bu^{n} \rangle \notag \\
		& + h \langle D^{-1} G(\bu^{n+1})\bu^{n+1},\bu^{n+1}-\bu^n \rangle \notag \\ & = h
		 \langle \left( G(\bu^{n+1})-G(\bu^{n})\right) D^{-1}\bu^{n+1}, \bu^{n+1}-\bu^{n} \rangle. 
		 \end{split}
	\]
Using 
	\begin{align*}
		\langle DC \bu^{n+1},\bu^{n+1}-\bu^{n} \rangle = &\tfrac{1}{2}\left(\langle DC \bu^{n+1},\bu^{n+1} \rangle - \langle DC \bu^{n},\bu^{n}\rangle \right) \\
		&+\tfrac{1}{2}  \langle DC (\bu^{n+1}-\bu^n),\bu^{n+1}-\bu^n  \rangle
	\end{align*}
we obtain
  \begin{align} 
	& \frac{1}{\Delta t} \|\bu^{n+1}-\bu^n\|_{D,h}^2+ \frac{h}{2} \langle DC \bu^{n+1},\bu^{n+1} \rangle \notag 
		 + h \langle D^{-1} G(\bu^{n+1})\bu^{n+1},\bu^{n+1}-\bu^n \rangle \notag  \notag \\ & \leq \frac{h}{2}\langle DC \bu^{n},\bu^{n} \rangle +
		 h\langle \left( G(\bu^{n+1})-G(\bu^{n})\right) D^{-1}\bu^{n+1},\bu^{n+1}-\bu^{n} \rangle. 
		 \label{ee} \end{align}
		 To control the term $\langle D^{-1} G(\bu^{n+1})\bu^{n+1},\bu^{n+1}-\bu^n \rangle$ we use a convexity argument. 
		 The function $f(y)=\sin^2(y)$ is convex on $[-\frac{\pi}{4}, \frac{\pi}{4}]$. Hence, for all $x,y \in [-\frac{\pi}{4}, \frac{\pi}{4}]$ we have $f(x) \geq f(y) + f'(y)(x-y)$. 
	 From $\| \bu^0 \|_\infty \leq \pi/4$ and  Lemma \ref{lemma3}, we have $\| \bu^{n} \|_\infty \leq \pi/4$ for all $n=0,\hdots,M$. Thus we get 
	\[
		\sin^2 (u^{n}_i) \geq \sin^2 (u^{n+1}_i) + \sin(2u_i^{n+1})\left(u_i^n-u_i^{n+1} \right),
	\]
	and
	\begin{align*}
		& \langle D^{-1} G(\bu^{n+1})\bu^{n+1},\bu^{n+1}-\bu^n \rangle =  \sum_{i=1}^N \frac{1}{x_i} \frac{\sin(2u_i^{n+1})}{2}(u_i^{n+1}-u_i^n) \\
		&\geq \sum_{i=1}^N \frac{1}{2x_i}\left(\sin^2(u_i^{n+1})-\sin^2(u_i^{n})\right) \notag \\
		&=\tfrac{1}{2}\left(\langle D^{-1} F(\bu^{n+1}),F(\bu^{n+1})\rangle - \langle D^{-1} F(\bu^{n}),F(\bu^{n})\rangle \right).
	\end{align*}
	Using this in \eqref{ee} we obtain
	\begin{align} 
	& \frac{2}{\Delta t} \|\bu^{n+1}-\bu^n\|_{D,h}^2+ \cE_h(\bu^{n+1}) \notag \\ 
		  & \leq \cE_h(\bu^{n}) +
		2 h\langle \left( G(\bu^{n+1})-G(\bu^{n})\right) D^{-1}\bu^{n+1},\bu^{n+1}-\bu^{n} \rangle. 
		 \label{eee} \end{align}
	It remains to control the last term in \eqref{eee}. Note:
	\begin{align*} &  h\langle \left( G(\bu^{n+1})-G(\bu^{n})\right) D^{-1}\bu^{n+1},\bu^{n+1}-\bu^{n} \rangle \\ & = h \sum_{i=1}^N \big(g(u_i^{n+1})-g(u_i^{n})\big) x_i^{-1} u_i^{n+1} (u_i^{n+1}-u_i^n).
	\end{align*}
	Recall that $x_i^{-\alpha} |u_i^n| \leq d_\alpha$ holds for all $n$, cf. \eqref{newlabel}.  Using this, Lemma~\ref{lemma8} and $x_i^{2(\alpha-1)} \leq h^{2(\alpha-1)}$ we obtain
	\begin{align*}   
	  & h\langle \left( G(\bu^{n+1})-G(\bu^{n})\right) D^{-1}\bu^{n+1},\bu^{n+1}-\bu^{n} \rangle \leq \tfrac43 d_\alpha^2h \sum_{i=1}^N x_i^{2\alpha -1} |u_i^{n+1}-u_i^n|^2 \\
	& \leq  \tfrac43 d_\alpha^2h \sum_{i=1}^N x_i^{2(\alpha -1)} x_i |u_i^{n+1}-u_i^n|^2 \leq \tfrac43 d_\alpha^2 h^{2(\alpha-1)} \|\bu^{n+1} - \bu^n\|_{D,h}^2.
\end{align*}
We consider two ways to proceed. In the first approach we   assume that $\frac43 d_\alpha^2 h^{2(\alpha-1)} \leq \frac{1}{\Delta t}$ holds. Using this we observe that the last term in  \eqref{eee}  can be absorbed in the term on the left hand side in \eqref{eee}.  This proves the result \eqref{diss2}. The second possibility is to assume $\alpha > \frac12$ and use the result of Theorem~\ref{mainthm}, i.e., $\|\bu^{n+1} - \bu^n\|_{D,h} \leq 2 c(\Delta t +h^2 |\ln h|^\frac12)$ and thus
\[
  h\langle \left( G(\bu^{n+1})-G(\bu^{n})\right) D^{-1}\bu^{n+1},\bu^{n+1}-\bu^{n} \rangle  \leq
  \tfrac{32}{3} d_\alpha^2 h^{2(\alpha-1)} c^2\big((\Delta t)^2 + h^4 |\ln h|\big).
\]
Using this in \eqref{eee} we obtain the estimate in \eqref{diss1}.
\end{proof}

We discuss the results \eqref{diss2} and \eqref{diss1}. In \eqref{diss2} we have  monotone discrete energy dissipation under the condition $\Delta t \leq \tfrac34 d_\alpha^{-2} h^{2(1-\alpha)}$. We make the reasonable assumption that $ d_\alpha$ is uniformly bounded, cf.~\eqref{destimate}. Hence, the condition is of the form $\Delta t \leq C h^{2(1-\alpha)}$. Assume a scaling property $\Delta t = h^{\beta}$ with $\beta >0$. Then this  condition is satisfied iff $\beta > 2(1-\alpha)$. For example, for $\alpha= \frac12$ we require (only) $\Delta t \sim h^\beta$ with $\beta >1$. In \eqref{diss1}  we have a perturbed energy dissipation. The perturbation is of the form $C\big(h^{2(\alpha-1)}(\Delta t)^2 +h^{2(\alpha +1)}|\ln h|\big)$. The term $h^{2(\alpha +1)}|\ln h|$ tends to $0$ for $h \downarrow 0$. For the other term to go to zero we need a condition that bounds $\Delta t$ in  terms of $h$. Consider again the scaling $\Delta t = h^{\beta}$. Then $h^{2(\alpha-1)}(\Delta t)^2 \to 0$ for $h \downarrow 0$ iff $\beta > 1 - \alpha$, which is a  weaker condition than the one used in \eqref{diss2}. 

Finally note that for $\alpha=0$, i.e., under the weakest assumption on the initial condition, namely $\|\bu^0\|_\infty \leq c_0=\frac{\pi}{2}$, we still have a monotone discrete energy dissipation as in \eqref{diss2}, provided $\Delta t \leq \tfrac34 \|\bu^0\|_\infty^{-2} h^2$ is satisfied.

\section{Numerical experiments}\label{sec:NumericalExp}
We consider a problem as in \eqref{PDE} with $\bu_0(x) = \pi(1-x)x$ and $T=0.1$.  In this case we have a globally smooth solution. We apply the method \eqref{discrete} and determine the errors at the end time point, i.e. $\|\bu_{\rm ref}^M-\bu^M\|_{D,h}$. The code can be found in \cite{nguyen_2025_15481333}.
\begin{remark} \rm
A sufficiently accurate reference solution $\bu_{\rm ref}$ is determined by using the scheme \eqref{discrete} with sufficiently small mesh and time step sizes. The accuracy is validated by comparing numerical solutions with those of a BDF2 version of \eqref{discrete}.
\end{remark}

The numerical solution is shown in Figure~\ref{fig:plot_solution}. 
\begin{figure}[ht!]
   \centering
	\includegraphics[scale=0.16]{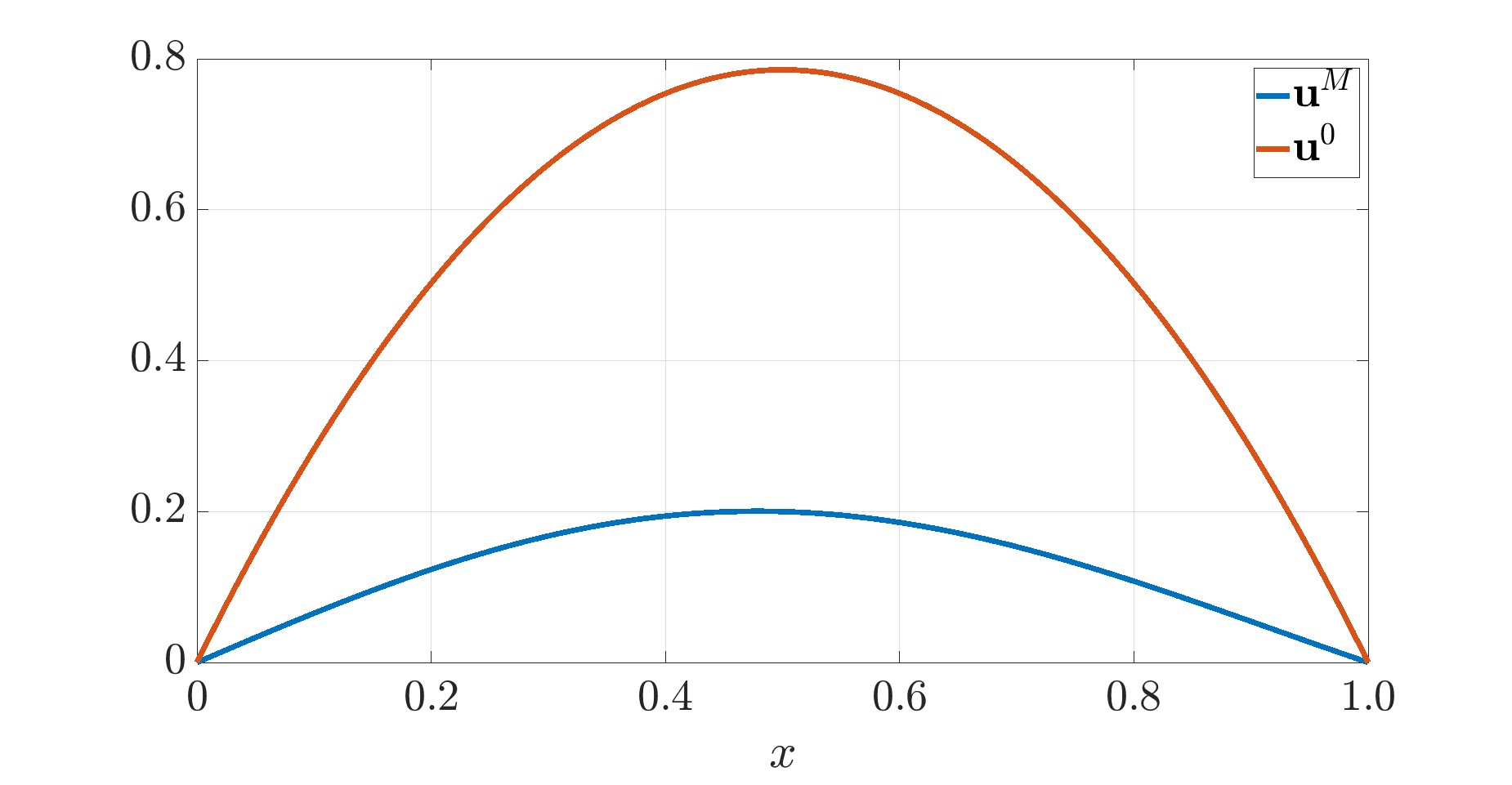}
	\caption{Solution $\bu^M$ of discrete problem \eqref{discrete} for $ h=10^{-3}, \Delta t = 10^{-6}$.}
	\label{fig:plot_solution}
\end{figure}
In
Table~\ref{table:SSHMHF_BDF1_conv_example_2} and Table~\ref{table:SSHMHF_h_conv_example_2} the errors for different mesh and time step sizes are shown. The  convergence orders with respect to the time step and the mesh size are as expected.
In  Table~\ref{table:SSHMHF_h_conv_example_2}, for $h=2^{-7}$ the error in the time discretization starts to influence the size of the discretization error,  and this is why  the convergence order is decreasing.
\begin{table}[ht!]
	\centering
	\begin{tabular}{ |p{2.5cm}|| p{2.65cm}| p{1cm}|}
		\hline
		$h = 10^{-3}$ & $\| \bu_{\text{ref}}^{M}-\bu^M\|_{D,h}$  & EOC \\
		\hline
		$\Delta t=10^{-2}$  & $1.0320e-02$ & $-$  \\
		$\Delta t=5\cdot 10^{-3}$  & $5.2699e-03$ & $0.97$ \\
		$\Delta t=2.5\cdot 10^{-3}$ & $2.6633e-03$ & $0.98$ \\
		$\Delta t=1.25\cdot 10^{-3}$ & $1.3389e-03$ & $0.99$ \\
		$\Delta t=6.25\cdot 10^{-4}$ & $6.7126e-04$ & $1.00$ \\
		\hline 
	\end{tabular}
	\caption{Discretization error for \eqref{discrete} }
	\label{table:SSHMHF_BDF1_conv_example_2}
\end{table}

\begin{table}[ht!]
	\centering
	\begin{tabular}{ |p{2.5cm}|| p{2.65cm}| p{1cm}|}
		\hline
		$\Delta t= 10^{-6}$ & $\| \bu_{\text{ref}}^{M}-\bu^M\|_{D,h}$  & EOC \\
		\hline 
		$h=2^{-2}$ & $6.7606e-03 $ & $-$ \\
		$h =2^{-3}$  & $1.6630e-03$ & $2.02$  \\
		$h=2^{-4}$  & $4.1413e-04$ & $2.00$ \\
		$h=2^{-5}$ & $1.0408e-04$ & $1.99$ \\
		$h=2^{-6}$ & $2.6716e-05$ & $1.96$ \\
		$h=2^{-7}$ & $7.3860e-06$ & $1.85$ \\
		\hline 
	\end{tabular}
	\caption{Discretization error for \eqref{discrete}}
	\label{table:SSHMHF_h_conv_example_2}
\end{table}

We performed an experiment to validate Lemma~\ref{lemma4}. For the initial condition $\bu_0(x) = \pi(1-x)x$ we have $\|\bu_0\|_\infty=\tfrac14 \pi=: c_\alpha $. We then have a corresponding $\alpha$ as defined in Lemma~\ref{lemma4} with value $\alpha = \sqrt{\frac{2}{\pi}}$. For this $\alpha$ value we show the results for $\|\bu_\alpha^n\|_\infty=\|D^{-\alpha}\bu^n\|_\infty$ in Figure~\ref{fig:D_alpha}.  We observe that the numerical solution is decreasing monotonically with respect to time $t$,  as predicted by Lemma \ref{lemma4}.
\begin{figure}[ht!]
    \centering
	\includegraphics[scale=0.16]{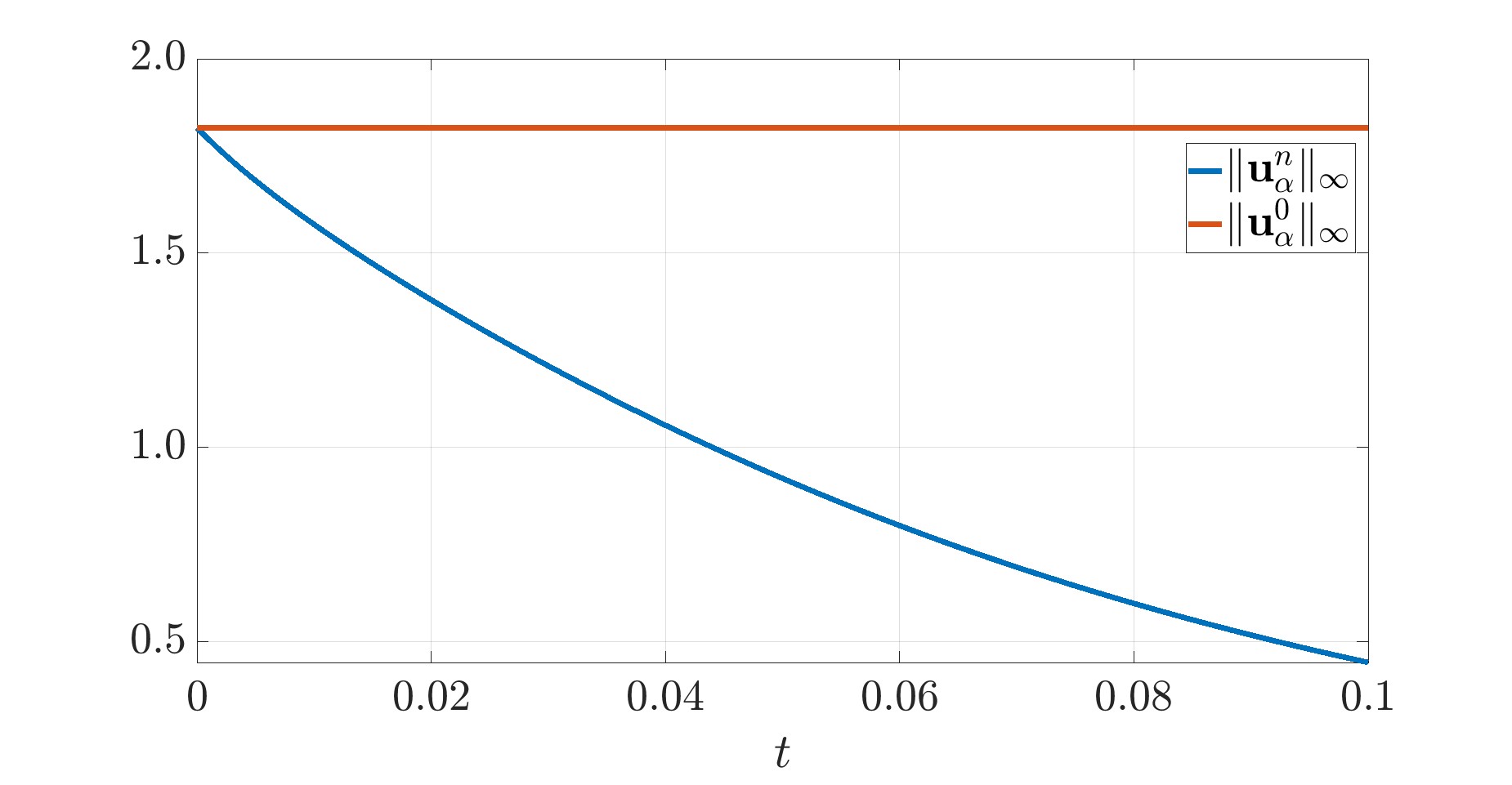}
	\caption{$\|D^{-\alpha}\bu^n\|_\infty$, $\alpha=\sqrt{\tfrac{2}{\pi}}$ and $\bu^n$ from \eqref{discrete}; $ h=10^{-3}, \Delta t = 10^{-6}$}
	\label{fig:D_alpha}
\end{figure}

We also computed the discrete energy values $\cE_h(\bu^n)$, cf. \eqref{discrete_energy}. We used mesh size and time step values $ h=10^{-3}$, $\Delta t = 10^{-6}$, for which the condition in \eqref{diss2} is satisfied. Results are shown in Figure~\ref{fig:Energy_smooth}. We observe a monotone decreasing behavior, consistent with the result in Lemma~\ref{energy_lemma}. Further experiments indicate that the condition in \eqref{diss2}  is not essential for the discrete energy decrease to hold. 
\begin{figure}[ht!]
    \centering
	\includegraphics[scale=0.16]{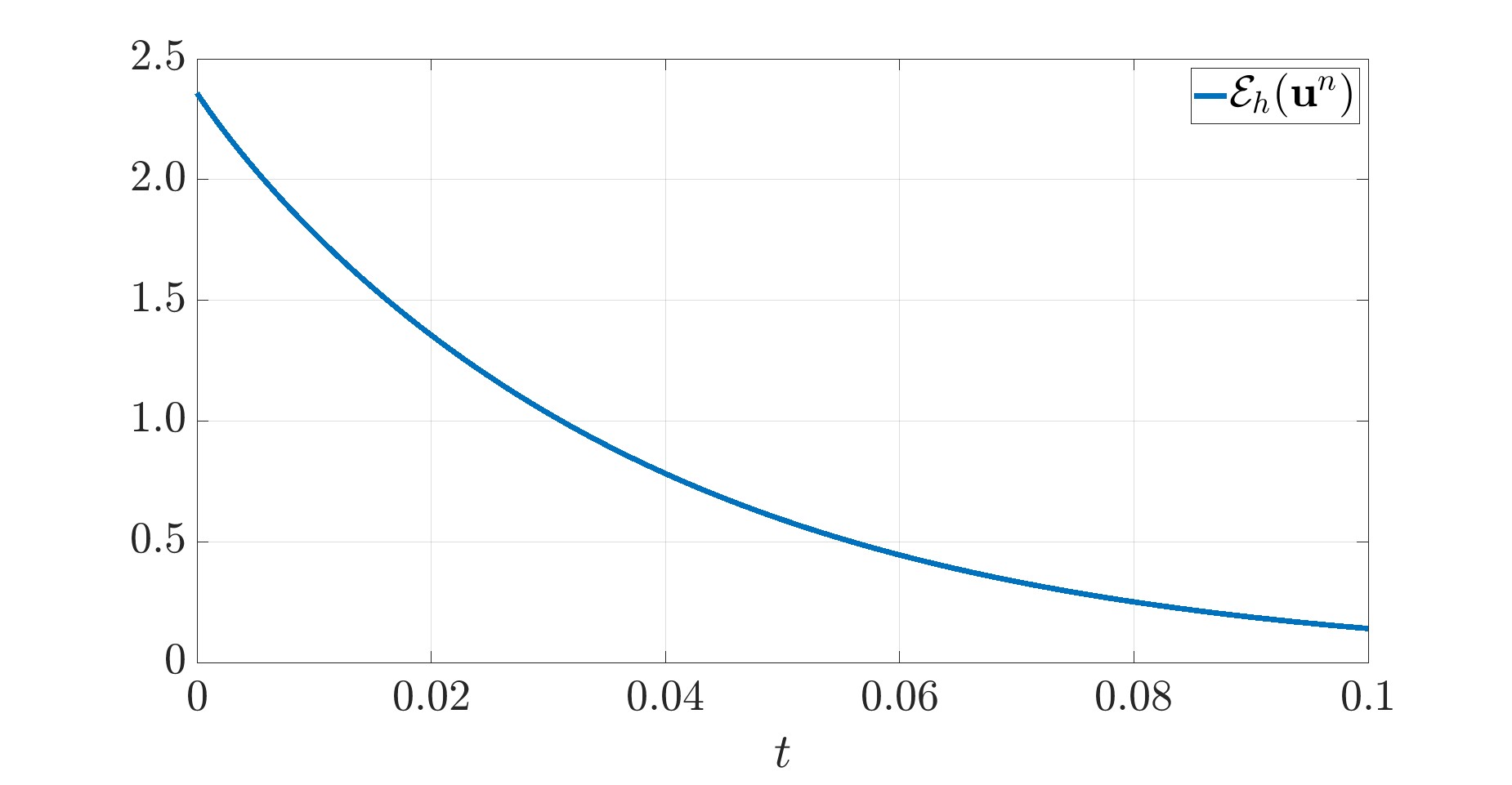}
	\caption{$\cE_h(\bu^n)$ for $\bu^n$ from \eqref{discrete};  $ h=10^{-3}$, $\Delta t = 10^{-6}$}
	\label{fig:Energy_smooth}
\end{figure}

Finally, although not covered by the analyis of this paper, we present a few results for an example in which finite time blow up is expected. We consider \eqref{PDE} with $\bu_0(x) = 9\pi(1-x)x$ and $T=0.01$. We then have $\bu_0(0.5)> \pi$, which according to \cite[Proposition 2.2]{Bertsch_DalPasso_vanderHout_2002}, may lead to blow up in finite time with energy decreasing in time. We discretize the problem using \eqref{discrete} with  $ h=10^{-3}, \Delta t = 10^{-6}$. The discrete solution at $t=T=M \Delta t$ is shown in Figure~\ref{fig:blowup_sol_h_1e-3_dt_1e-6_bdf1}.
\begin{figure}[ht!]
    \centering
	\includegraphics[scale=0.16]{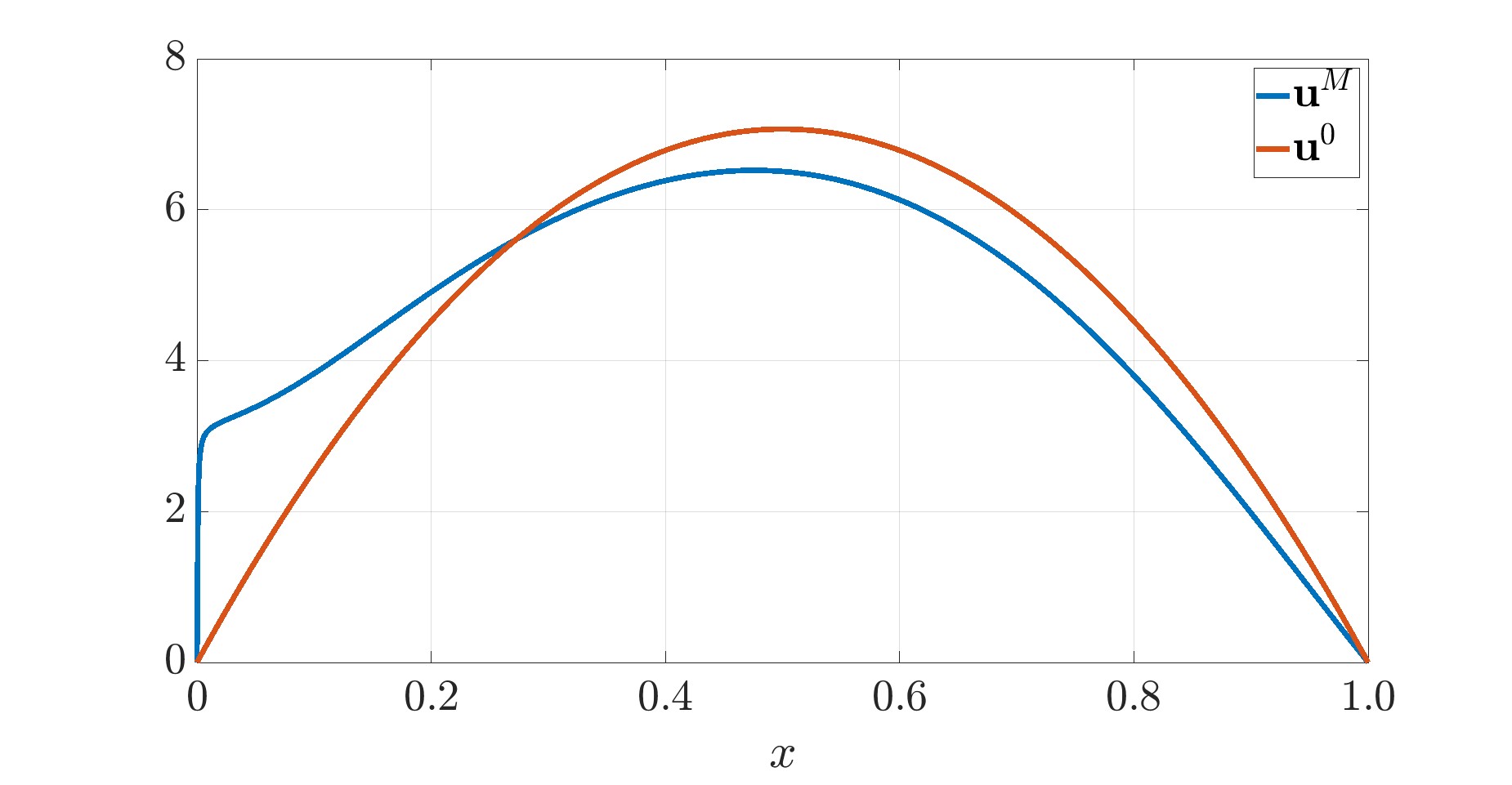}
	\caption{$\bu_0(x) = 9\pi(1-x)x$, $T=0.01$. Discrete solution $\bu^n$ from \eqref{discrete}; $ h=10^{-3}, \Delta t = 10^{-6}$.}
	\label{fig:blowup_sol_h_1e-3_dt_1e-6_bdf1}
\end{figure}

The result suggests that indeed  a finite-time blow up will occur, which agrees with \cite[Section 4.3 and Section 4.4]{HaynesHuangZegeling2013} for a similar initial condition with the property that $\bu_0(x)>\pi$ for some $x \in (0,1)$.
The discrete energy is decreasing,  cf.~Figure \ref{fig:energy_dissipation_blowup}.
\begin{figure}[ht!]
    \centering
	\includegraphics[scale=0.16]{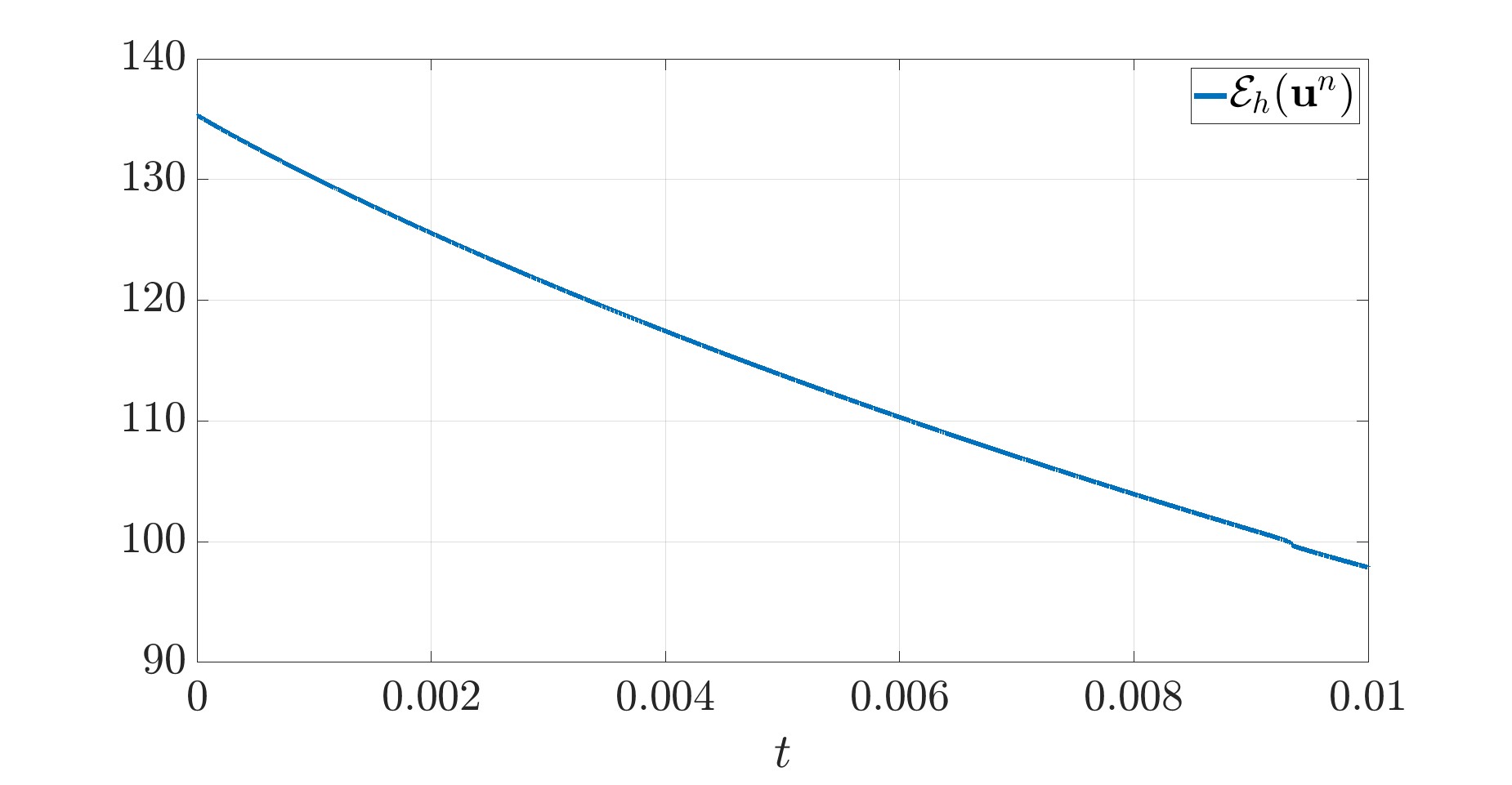}
	\caption{$\bu_0(x) = 9\pi(1-x)x$; $\cE_h(\bu^n)$  for $\bu^n$ from \eqref{discrete}; $ h=10^{-3}, \Delta t = 10^{-6}$.}
	\label{fig:energy_dissipation_blowup}
\end{figure}
 We observe in Figure \ref{fig:D_alpha_h_1e-3_blowup} that $\|D^{-1} \bu^n\|_\infty$  is monotonically \emph{increasing}. There is also a steep increase of $\|D^{-1} \bu^n\|_\infty$ at $t \approx 0.9 \cdot 10^{-2}$ which is an indication of  blow up of the solution close to this point in  time. 
 
\begin{figure}[ht!]
    \centering
	\includegraphics[scale=0.16]{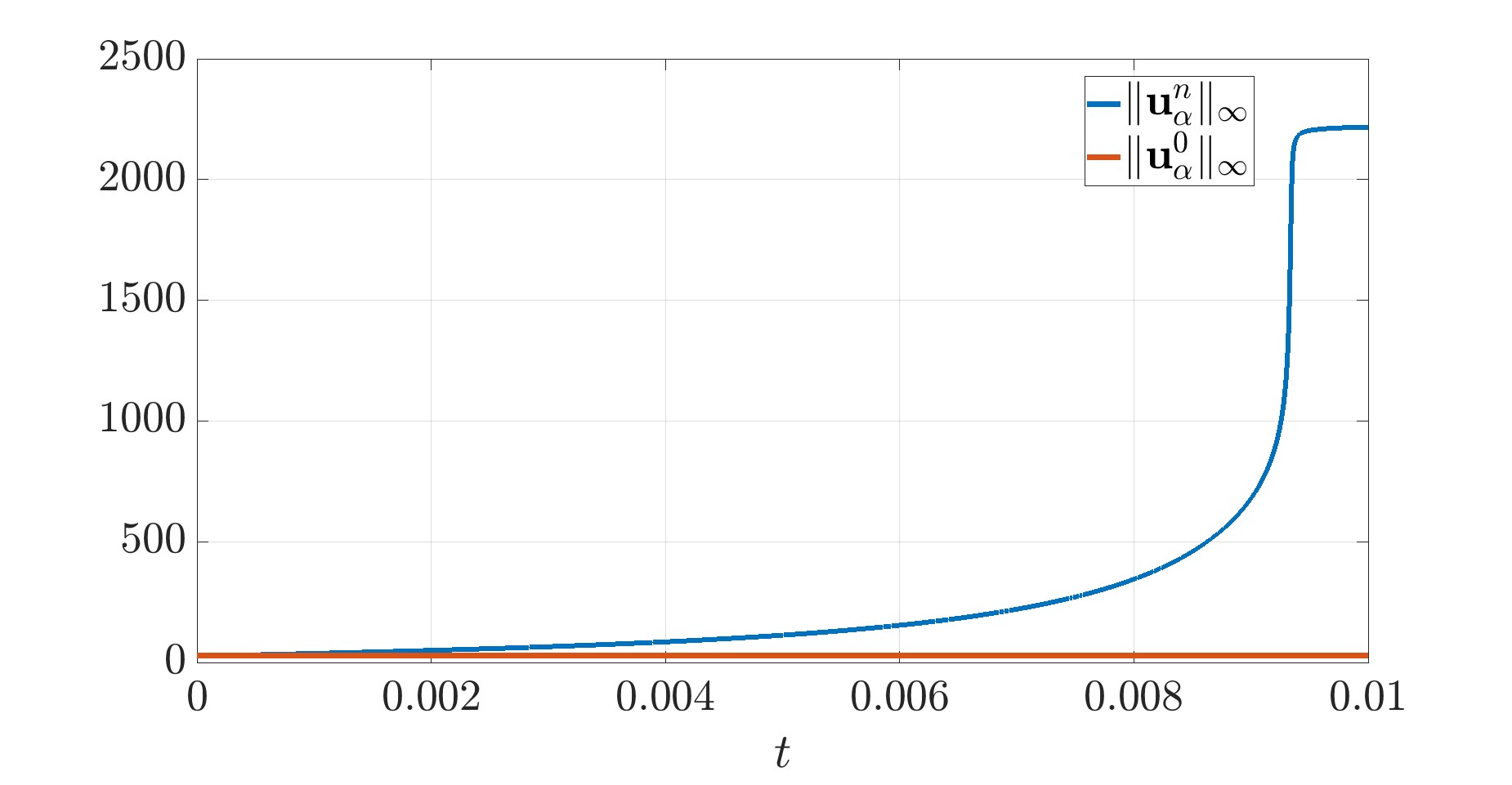}
	\caption{$\bu_0(x) = 9\pi(1-x)x$; $\alpha = 1$ for $\bu^n$ from \eqref{discrete}; $h=10^{-3}$, $\Delta t= 10^{-6}$.}
	\label{fig:D_alpha_h_1e-3_blowup}
\end{figure}
These results suggest that this is an  example of a blow up behavior as described in \cite[Proposition 2.2]{Bertsch_DalPasso_vanderHout_2002}. A  thorough numerical investigation of blow behavior in (radially symmetric) HMHF problems is left for future work. 

\ \\[10ex]
{\bf Acknowledgements} The authors acknowledge funding by the Deutsche Forschungsgemeinschaft (DFG, German Research Foundation) – project number 442047500 – through the Collaborative Research Center “Sparsity and Singular
Structures” (SFB 1481).}
\bibliographystyle{siam}
\bibliography{literature}{}

\end{document}